\documentclass[11pt]{amsart}
\input{preamble}

\let\oldeqref\eqref
\makeatletter
\RenewDocumentCommand\eqref{s m}{%
  \IfBooleanTF#1%
  {\textup{\tagform@{\ref*{#2}}}}
  {\oldeqref{#2}}
}
\makeatother

\newcommand{\Bir}{{\operatorname{Bir}}}

\newcommand{\sC}{{\mathcal C}}
\newcommand{\sD}{{\mathcal D}}

\newcommand{\sF}{{\mathcal F}}

\newcommand{\sH}{{\mathcal H}}

\newcommand{\sM}{{\mathcal M}}

\newcommand{\sS}{{\mathcal S}}

\newcommand{\sX}{{\mathcal X}}

\newcommand{\N}{{\mathbb N}}

\newcommand{\Q}{{\mathbb Q}}

\newcommand{\Z}{{\mathbb Z}}

\newcommand{\Aut}{\operatorname{Aut}}

\newcommand{\NS}{\operatorname{NS}}

\newcommand{\Pic}{\operatorname{Pic}}

\newcommand{\supp}{\operatorname{supp}}

\renewcommand{\to}[1][]{\xrightarrow{\ #1\ }}

\makeatletter
\newcommand*{\da@rightarrow}{\mathchar"0\hexnumber@\symAMSa 4B }
\newcommand*{\da@leftarrow}{\mathchar"0\hexnumber@\symAMSa 4C }
\newcommand*{\xdashrightarrow}[2][]{%
  \mathrel{%
    \mathpalette{\da@xarrow{#1}{#2}{}\da@rightarrow{\,}{}}{}%
  }%
}
\newcommand{\xdashleftarrow}[2][]{%
  \mathrel{%
    \mathpalette{\da@xarrow{#1}{#2}\da@leftarrow{}{}{\,}}{}%
  }%
}
\newcommand*{\da@xarrow}[7]{%
  \sbox0{$\ifx#7\scriptstyle\scriptscriptstyle\else\scriptstyle\fi#5#1#6\m@th$}%
  \sbox2{$\ifx#7\scriptstyle\scriptscriptstyle\else\scriptstyle\fi#5#2#6\m@th$}%
  \sbox4{$#7\dabar@\m@th$}%
  \dimen@=\wd0 %
  \ifdim\wd2 >\dimen@
    \dimen@=\wd2 %
  \fi
  \count@=2 %
  \def\da@bars{\dabar@\dabar@}%
  \@whiledim\count@\wd4<\dimen@\do{%
    \advance\count@\@ne
    \expandafter\def\expandafter\da@bars\expandafter{%
      \da@bars
      \dabar@ 
    }%
  }%
  \mathrel{#3}%
  \mathrel{%
    \mathop{\da@bars}\limits
    \ifx\\#1\\%
    \else
      _{\copy0}%
    \fi
    \ifx\\#2\\%
    \else
      ^{\copy2}%
    \fi
  }%
  \mathrel{#4}%
}
\makeatother

\newtheoremstyle{citing}
  {}
  {}
  {\itshape}
  {}
  {\bfseries}
  {\textbf{.}}
  {.5em}
  {\thmnote{#3}}

\theoremstyle{plain}

\newtheorem{theorem}[subsection]{Theorem}

\newtheorem{lemma}[subsection]{Lemma}
\newtheorem{corollary}[subsection]{Corollary}
\newtheorem{regprin}[subsection]{Regeneration principle}

\newtheorem{proposition}[subsection]{Proposition}

\theoremstyle{remark}

\theoremstyle{definition}

\newtheorem{definition}[subsection]{Definition}

\newtheorem{hypothesis}[subsection]{Hypothesis}
\numberwithin{equation}{section}

\theoremstyle{remark}
\newtheorem{remark}[subsection]{Remark}

{\theoremstyle{citing}
}

\makeatletter
\newsavebox\myboxA
\newsavebox\myboxB
\newlength\mylenA

\newcommand*\xtilde[2][0.8]{%
    \sbox{\myboxA}{$\m@th#2$}%
    \setbox\myboxB\null
    \ht\myboxB=\ht\myboxA%
    \dp\myboxB=\dp\myboxA%
    \wd\myboxB=#1\wd\myboxA
    \sbox\myboxB{$\m@th\widetilde{\copy\myboxB}$}
    \setlength\mylenA{\the\wd\myboxA}
    \addtolength\mylenA{-\the\wd\myboxB}%
    \ifdim\wd\myboxB<\wd\myboxA%
       \rlap{\hskip 0.5\mylenA\usebox\myboxB}{\usebox\myboxA}%
    \else
        \hskip -0.5\mylenA\rlap{\usebox\myboxA}{\hskip 0.5\mylenA\usebox\myboxB}%
    \fi}

\newbox\usefulbox

\def\getslant #1{\strip@pt\fontdimen1 #1}

\def\xxtilde #1{\mathchoice
 {{\setbox\usefulbox=\hbox{$\m@th\displaystyle #1$}%
    \dimen@ \getslant\the\textfont\symletters \ht\usefulbox
    \divide\dimen@ \tw@ 
    \kern\dimen@ 
    \xtilde{\kern-\dimen@ \box\usefulbox\kern\dimen@ }\kern-\dimen@ }}
 {{\setbox\usefulbox=\hbox{$\m@th\textstyle #1$}%
    \dimen@ \getslant\the\textfont\symletters \ht\usefulbox
    \divide\dimen@ \tw@ 
    \kern\dimen@ 
    \xtilde{\kern-\dimen@ \box\usefulbox\kern\dimen@ }\kern-\dimen@ }}
 {{\setbox\usefulbox=\hbox{$\m@th\scriptstyle #1$}%
    \dimen@ \getslant\the\scriptfont\symletters \ht\usefulbox
    \divide\dimen@ \tw@ 
    \kern\dimen@ 
    \xtilde{\kern-\dimen@ \box\usefulbox\kern\dimen@ }\kern-\dimen@ }}
 {{\setbox\usefulbox=\hbox{$\m@th\scriptscriptstyle #1$}%
    \dimen@ \getslant\the\scriptscriptfont\symletters \ht\usefulbox
    \divide\dimen@ \tw@ 
    \kern\dimen@ 
    \xtilde{\kern-\dimen@ \box\usefulbox\kern\dimen@ }\kern-\dimen@ }}%
 {}}

\newcommand*\xoverline[2][0.75]{%
    \sbox{\myboxA}{$\m@th#2$}%
    \setbox\myboxB\null
    \ht\myboxB=\ht\myboxA%
    \dp\myboxB=\dp\myboxA%
    \wd\myboxB=#1\wd\myboxA
    \sbox\myboxB{$\m@th\overline{\copy\myboxB}$}
    \setlength\mylenA{\the\wd\myboxA}
    \addtolength\mylenA{-\the\wd\myboxB}%
    \ifdim\wd\myboxB<\wd\myboxA%
       \rlap{\hskip 0.5\mylenA\usebox\myboxB}{\usebox\myboxA}%
    \else
        \hskip -0.5\mylenA\rlap{\usebox\myboxA}{\hskip 0.5\mylenA\usebox\myboxB}%
    \fi}

\def\xxoverline #1{\mathchoice
 {{\setbox\usefulbox=\hbox{$\m@th\displaystyle #1$}%
    \dimen@ \getslant\the\textfont\symletters \ht\usefulbox
    \divide\dimen@ \tw@ 
    \kern\dimen@ 
    \overline{\kern-\dimen@ \box\usefulbox\kern\dimen@ }\kern-\dimen@ }}
 {{\setbox\usefulbox=\hbox{$\m@th\textstyle #1$}%
    \dimen@ \getslant\the\textfont\symletters \ht\usefulbox
    \divide\dimen@ \tw@ 
    \kern\dimen@ 
    \xoverline{\kern-\dimen@ \box\usefulbox\kern\dimen@ }\kern-\dimen@ }}
 {{\setbox\usefulbox=\hbox{$\m@th\scriptstyle #1$}%
    \dimen@ \getslant\the\scriptfont\symletters \ht\usefulbox
    \divide\dimen@ \tw@ 
    \kern\dimen@ 
    \xoverline{\kern-\dimen@ \box\usefulbox\kern\dimen@ }\kern-\dimen@ }}
 {{\setbox\usefulbox=\hbox{$\m@th\scriptscriptstyle #1$}%
    \dimen@ \getslant\the\scriptscriptfont\symletters \ht\usefulbox
    \divide\dimen@ \tw@ 
    \kern\dimen@ 
    \xoverline{\kern-\dimen@ \box\usefulbox\kern\dimen@ }\kern-\dimen@ }}%
 {}}
\makeatother

{\theoremstyle{definition}

\title{Regenerations and applications}

\author{Giovanni Mongardi}
\address{Giovanni Mongardi\\Alma Mater Studiorum, Università di Bologna,  P.zza di porta san Donato, 5, 40126 Bologna, Italia}
\email{giovanni.mongardi2@unibo.it}

\author{Gianluca Pacienza}
\address{Gianluca Pacienza\\ 
Universit\'e de Lorraine, CNRS, IECL\\
F-54000 Nancy -- France }
\email{gianluca.pacienza@univ-lorraine.fr}

\begin{document}

\begin{abstract}
Chen-Gounelas-Liedtke recently introduced a powerful regeneration technique, a process opposite to specialization,  to prove existence results for rational curves on projective $K3$ surfaces. We show that,  for projective irreducible holomorphic symplectic manifolds,  an analogous regeneration 
principle   holds and provides a very flexible tool to prove existence of uniruled divisors, significantly improving known results. 
\end{abstract}

\subjclass[2020]{14H45, 14J42 (primary).}

\keywords{rational curves, irreducible holomorphic symplectic manifolds.\\
}

\maketitle
\section{Introduction}\label{s:intro}
Rational curves on $K3$ surfaces have now been studied for decades, with motivations also coming from arithmetic geometry, (non-)hyperbolicity questions, and general conjectures on $0$-cycles. A natural generalization of $K3$ surfaces is given by irreducible holomorphic symplectic (IHS) manifolds, which are compact, simply connected K\"ahler manifolds with $H^{2,0}$ generated by a symplectic form. In any even dimension $2n$, $n\geq 2$, there are two known deformation classes (cf. \cite{Bea83}), one is given by Hilbert schemes of points on $K3$ surfaces and their deformations (called varieties of $K3^{[n]}$ type), and the other is given by deformations of an analogous construction using abelian surfaces (called varieties of generalized Kummer type). Two more  deformation classes discovered by O'Grady exist in dimension 6 and 10 (cf. \cite{OG99,OG03}). For the basic theory of IHS manifolds we refer the reader to e.g. \cite{Bea83,Huy99}.\\
In recent years rational curves on projective IHS manifolds have been actively investigated  with different objectives and techniques, cf. e.g. \cite{AV15, BHT15, Vo16, LP15, MP, CMP19, KLM19b, OSY19, Ber20} and the references therein. 
Rational curves covering a divisor on an IHS manifold behave very well with respect to deformation theory, i.e. they deform in their Hodge locus inside the parameter space of  deformations of the IHS manifold and keep covering a divisor. This has been one of the main properties used to prove existence results and, at the same time, one of the main limitations. Indeed, to 
produce a uniruled divisor in an ample linear system of a polarized IHS $(X,H)$ one would try and  exhibit such an example on a special point $(X_0,H_0)$ in the same connected  component of the corresponding moduli space. 
As proved in \cite{OSY19} in some cases it is impossible to do it with primitive rational curves. On the other hand in \cite{CMP19, MP, MPerr} this approach was successfully implemented to prove that outside at most a finite number  of connected components (precisely those not satisfying the necessary conditions given in \cite{OSY19}) of the moduli spaces 
of projective IHS manifolds of 
$K3^{[n]}$ or generalized Kummer type, for all the corresponding points $(X,H)$
there
exists a positive integer $m$ such that the linear system $|mH|$ contains a uniruled divisor covered
by rational curves of primitive homology class Poincaré-dual to that of $|H|$. For a completely different proof (based on Gromov-Witten theory) of the existence of uniruled divisors covered
by primitive rational curves on deformations of $K3^{[n]}$ see \cite[Theorem 0.1]{OSY19}. Due to the cases left out by \cite{CMP19, OSY19}, respectively \cite{MP, MPerr}, one could reasonably wonder whether uniruled divisors on such manifolds do always exist.

More recently Chen-Gounelas-Liedtke introduced in \cite{CGL} a new viewpoint to prove existence results for rational curves on projective $K3$ surfaces: regeneration, a process opposite to specialization. In this article we show that,  for projective irreducible holomorphic symplectic manifolds, an analogous regeneration 
principle   holds for uniruled divisors and provides a new and flexible tool to prove existence results. Combining this new viewpoint with results from \cite{CMP19, MP, MPerr} 
we are able to improve
significantly the available results, in some cases passing from no known existence result at all to density of uniruled divisors in the classical topology. 

To state our results we start with the following. 
\begin{definition}
    Let $\sX\to B$ be a family of IHS manifolds over a connected base. Let $0\in B$ and let $X_0$ be the corresponding fibre. Let $D_0\subset X_0$ be an integral uniruled divisor. A {\it regeneration} $\sD\subset \sX$ of $D_0$ is a flat family of uniruled and generically integral divisors $\sD\to B$ such that $D_0$ is a component of the fiber $\sD_0$ of $\sD$ over $0$.
\end{definition}
A reducible divisor is called uniruled if all of its components are.

\begin{hypothesis}\label{hyp:curve che rigano}
    Let $X$ be a projective IHS manifold. There exists a constant $d\geq 0$ such that  all primitive ample curve classes $[C]\in H_2(X,\Z)$ satisfying $q(C)> d$   have a representative $R\in [C]$ such that $R$ rules a prime divisor of class proportional to $[C]^\vee$.
\end{hypothesis}
Here, $[C]^\vee$ denotes the divisor $[D]\in \NS(X)\otimes\Q$ such that $C\cdot E=q(D,E)$ for all divisors $E$, where $q$ is the Beauville-Bogomolov-Fujiki form  on $X$. A curve is said ample if its dual divisor is ample. Analogously, we define the curve dual to a divisor.

The above hypothesis, which may look slightly unnatural, is the higher dimension analogue of \cite[Theorem A.1]{CGL} and, as we will see below, can be shown to hold  for IHS manifolds of $K3^{[n]}$ and generalized Kummer type, thanks to previous work done in \cite{CMP19,MP,MPerr}. 
Our main novel contribution is the following result which, despite the simplicity of its proof, seems to provide the right viewpoint to tackle these kind of questions.   

\smallskip
\begin{regprin}\label{regprin}
 Let $\sX\to B$ be a family of projective IHS manifolds with a central fibre $\sX_0$ satisfying hypothesis \ref{hyp:curve che rigano}.
 Let $D_0\subset \sX_0$ be an integral uniruled divisor on the central fibre. Then $D_0$ admits a regeneration. 
\end{regprin}

The regeneration principle works perfectly on IHS manifold of $K3^{[n]}$ or generalized Kummer type.

\begin{theorem}\label{thm:anyreg}
    Any integral uniruled divisor in a fiber of any family of projective IHS manifolds of $K3^{[n]}$ or generalized Kummer type admits a regeneration.
\end{theorem}
Our first application is to show existence of ample uniruled divisors also for the connected components of the moduli spaces left out by \cite{CMP19, MP, MPerr}.

\begin{theorem}\label{thm:to zal}
    Let $(X,H)$ be a polarized IHS manifold of $K3^{[n]}$ or generalized Kummer type, then there exists $m\in\N$ and a uniruled divisor in $|mH|$.
\end{theorem}
In particular the applications to zero-cycles pointed out in \cite[Theorems 1.7 and 1.8]{CMP19} now hold for all polarized IHS manifolds of $K3^{[n]}$ or generalized Kummer type. 

At the very general point in the $K3^{[n]}$-case we can drastically improve Theorem \ref{thm:to zal}. 
\begin{theorem}\label{thm:density}
    Let $\sM$ be an irreducible component of the moduli space of polarized IHS manifolds of $K3^{[n]}$-type. Then any polarized IHS manifold  $X$ outside a possibly countable union of subvarieties of $\sM$ verifies the following: any pair of points $x_1,x_2\in X$ can be arbitrarily approximated by a chain of at most $2n$ irreducible rational curves, each of which deforms in a family covering a divisor.
\end{theorem}
The above result can be seen as an effective non-hyperbolicity statement. The study of non-hyperbolicity of IHS manifolds dates back to Campana \cite{Cam92}, with more recent important contributions by Verbitsky \cite{Ver15} and Kamenova-Lu-Verbitsky \cite{KLV14}. We refer the interested reader to \cite{KL22} for a thorough discussion and a complete list of references. 

We can also show the following less strong but more precise result, which was previously known only in dimension 2 by \cite[Theorem 4.10]{BT00}.
\begin{theorem}\label{thm:bir}
    Let $X$ be a projective IHS manifold of $K3^{[n]}$ or Kummer type such that $Bir(X)$ is infinite. Then $X$ has infinitely many uniruled divisors.
\end{theorem}

We hope that this new viewpoint via regenerations could also lead to progress towards the existence of higher codimension algebraically coisotropic subvarieties.

{\bf Acknowledgements.} We thank Claire Voisin for suggesting to apply the regeneration principle to non-hyperbolicity questions and G. Ancona, Ch. Lehn and K. O'Grady for useful comments on a preliminary version. 
G.M. was supported by PRIN2020 research grant ”2020KKWT53”, by PRIN2022 research grant "2022PEKYBJ" and is a member
of the INDAM-GNSAGA. G.P. was supported by the CNRS International Emerging Actions (IEA)
project “Birational and arithmetic aspects of orbifolds”.

\section{Regenerations}\label{s:reg}

\begin{proof}[Proof of the Regeneration principle \ref{regprin}]
We can suppose that $\sX_0$ has Picard rank at least two and that $D_0$ is not proportional to the polarization, otherwise by \cite[Corollary 3.5]{CMP19}, we can deform a curve ruling $D_0$ over all of $B$, and obtain in this way a regeneration of $D_0$. 

Let $C_0$ be the class of a minimal curve ruling $D_0$. Let $\sH\in \Pic(\sX)$ be a relative polarization and $H_0$ its restriction to the central fibre $\sX_0$. Let $H_0^\vee$ be the (ample) class of a curve dual to $H_0$. We can choose $m\in\N$ big enough so that $mH_0^\vee-C_0$ is ample, primitive and of square bigger than $d_0$.
Therefore, by Hypothesis \ref{hyp:curve che rigano}, we have a rational curve $R_0\in [mH_0^\vee-C_0]$ which rules an ample divisor $F_0$ inside $\sX_0$.

As the divisor $F_0$ is ample 
 we have $C_0\cdot F_0>0$. Hence, we can fix a point in $C_0\cap F_0$ and pick a curve $R_0$ in the ruling of $F_0$ passing through this point. Notice that $C_0$ cannot coincide with the ruling of $F_0$, as $C_0$ and $R_0$ are not proportional (because the divisors they rule are not). In this way we obtain a connected rational curve of class $[C_0+R_0]$. By abuse of notation, we denote this curve by $C_0+R_0$. By \cite[Corollary 6.3]{Ber20}, which generalizes \cite[Corollary 3.5]{CMP19} to the reducible case, the curve $C_0+R_0$ deforms in its Hodge locus $\rm{Hdg_{[C_0+R_0]}}$ of the class $[C_0+R_0]=[mH_0^\vee]$ and keeps ruling a divisor on each point of $\rm{Hdg_{[C_0+R_0]}}$. By construction, this Hodge locus  coincides with $B$, as $C_0+R_0$ is a multiple of $H_0^\vee$, and the result follows.
\end{proof}

The following can be seen as a concentration of some of the main contributions of \cite{CMP19, MP, MPerr}, namely the study of the monodromy orbits, constructions of examples and deformation theory.  

\begin{proposition}\label{prop:hyp su K3}
    Hypothesis \ref{hyp:curve che rigano} holds for any family of manifolds of $K3^{[n]} $ and Kummer type, and the constant $d_0$ is $ (2n-2)^2(n-1)$ and $(2n+2)^2(n+1)$ respectively.
\end{proposition}
\begin{proof}
Let $(S,h_S)$ be a polarized K3 of genus $p$ and $(A,h_A)$  a polarized abelian surface of type $(1,p-1)$. 
We denote by $r_n$  the class of an exceptional rational curve which is the general fiber of the Hilbert-Chow morphism $S^{[n]}\to S^{(n)}$ (resp. $K_n(A)\subset A^{[n+1]}\to A^{(n+1)}$) and by $h_S\in H_2(S^{[n]}, \mathbb Z)$ (resp. $h_A\in H_2(K_n(A), \mathbb Z)$ ) the image of the class $h_S\in H_2(S, \mathbb Z)$ (resp. $h_A\in H_2(A, \mathbb Z)$) under the inclusion $H_2(S, \mathbb Z)\hookrightarrow H_2(S^{[n]}, \mathbb Z)$ (resp. $H_2(A, \mathbb Z)\hookrightarrow H_2(K_n(A), \mathbb Z)$). Recall that $q(h_S)=2p-2=q(h_A)$ and $q(r_n)$ equals $1/(2n-2)$ in the $K3^{[n]}$ case and $1/(2n+2)$ in the Kummer case.
 
 We take a primitive ample curve class $C\in H_2(X,\Z)$ such that $q(C)>n-1$ (resp. $n+1$ for Kummer type).
By \cite[Corollary 2.8]{CMP19} and \cite[Theorem 4.2]{MP}, the pair $(X,C)$ is deformation equivalent to the pair $(S^{[n]},h_S-2gr_n)$ with $2g\leq n-1$ or $(S^{[n]},h_S-(2g-1)r_n)$ with $2g\leq n$  (resp. $(K_n(A),h_A-2gr_n)$ or $(K_n(A),h_a-(2g-1)r_n)$ with $2g\leq n-1$ ). 

If $p\leq g$, we would get a contradiction since
$$
n-1 \leq q(C)=
q(h_S) -4g^2 \frac{1}{2(n-1)} =  2(p-1) -4g^2 \frac{1}{2(n-1)}\leq 2(g-1)-4g^2 \frac{1}{2(n-1)}\leq n -2
$$ 
 Therefore, $p\geq g$ and by \cite[Section 4.1]{CMP19} and \cite[Proof of Proposition 2.1]{MPerr}, the curves we obtain in $S^{[n]}$ (resp. in $K_n(A)$) have a rational representative which covers a divisor by \cite[Proposition 4.1]{CMP19} and \cite[Proposition 1.1]{MPerr}.  Such divisor then deforms in its Hodge locus by \cite[Proposition 3.1]{CMP19}, and the proposition follows.
\end{proof}

\begin{proof}[Proof of Theorem 
    \ref{thm:anyreg}]
    The result follows immediately from the combination of Proposition \ref{prop:hyp su K3} and the Regeneration principle \ref{regprin}.
\end{proof}

\section{Applications}\label{s:app}
In this section we provide the proofs of the applications of the Regeneration principle 
to IHS manifolds of $K3^{[n]}$-type or generalized Kummer type. 

\begin{proof}[Proof of Theorem \ref{thm:to zal}]
Again the result follows from the combination of Proposition \ref{prop:hyp su K3} and the Regeneration principle \ref{regprin}.  Indeed, suppose that $(X,H)$ is a polarized IHS manifold of $K3^{[n]}$-type and let us consider a connected component $\sM$ of the moduli space of polarized IHS manifolds containing $(X,H)$. By \cite[Theorem 2.5]{CMP19}, 
 there exists a point in $\sM$ which parametrizes the Hilbert scheme over a very general projective $K3$ $(S,H_S)$. Let us choose any rational curve $C$ in $S$, whose existence is guaranteed by Bogomolov-Mumford \cite{MM82}, see also \cite[Section VIII.23]{BHPV04}, and let us consider the uniruled divisor $D_C=\{Z\in S^{[n]}$ such that $supp(Z)\cap C\neq \emptyset\}$. We then apply the Regeneration principle \ref{regprin} to $D_C$, and obtain a regeneration of it on all IHS manifolds corresponding to points of $\sM$. As the very general element of $\sM$ has Picard rank one, the class of this regeneration is proportional to this unique class, hence our regeneration has class $mH$ on $X$, for some $m$.    For the generalized Kummer type we proceed the same way, by using  \cite[Theorem 4.2]{MP}  and \cite[Theorem 1.1]{KLM19} instead of the analogous results in the $K3^{[n]}$-type case. 
\end{proof}

More generally, we have the following result.

\begin{proposition}\label{prop:square zero}
Let $(X,H)$ be a projective IHS manifold of $K3^{[n]}$ or Kummer type, and let $D\in \Pic(X)$ be a divisor with $q(D)\geq 0$ and $(D,H)>0$. Then there exists a uniruled divisor in $|mD|$ for some $m\in \N$.
\end{proposition}
\begin{proof}
    The proof is analogous to Theorem \ref{thm:to zal}, with an extension to the case of square zero classes.
    If $D$ has positive square, instead of the moduli space of polarized IHS manifolds we consider the moduli space $\sM$ of lattice polarized IHS manifolds such that $\Pic(X)$ contains a divisor of square $q(D)$, and pick the connected component containing $(X,D)$. Let us choose a parallel transport operator $\gamma$ on $\sM$ such that $\gamma(X)$ has Picard rank $1$. Therefore, $\gamma(D)$ is ample on $\gamma(X)$. By Theorem \ref{thm:to zal}, a multiple of $\gamma(D)$ is uniruled by a rational curve $\gamma(C)$, which has class proportional to $\gamma(D)^\vee$. Therefore by \cite[Proposition 3.1]{CMP19}, $\gamma(C)$ deforms in its Hodge locus, which by construction contains $(X,D)$ and we obtain a rational curve $C$ covering a multiple of $D$.
    If $q(D)=0$, we can suppose that $D$ is nef by \cite[Proposition 5.6]{Mar11}, otherwise we follow the same reasoning as above to reduce to the nef case.     As $X$ is projective, we have an ample divisor $H\in \Pic(X)$. Let $L$ be the saturated lattice generated by $D$ and $H$, and let us consider the component $\sM$ of the moduli space of $L$ lattice polarized IHS manifolds containing $(X,L)$. Inside of $\sM$, by \cite[Theorem 3.13]{MPdens} we can pick a point $\gamma(X)$ such that $\gamma(D)$ stays nef and there exists a prime exceptional divisor $E$ on $\gamma(X)$ such that $q(\gamma(D),E)>0$\footnote{By the above cited theorem, the locus where a given extra class is algebraic is dense in $\mathcal{M}$, and the locus where this class $E$ has a fixed intersection with $\gamma(D)$ is a proper Zariski closed subset of $\mathcal{M}$, therefore the locus where the intersection is positive is non-empty.}. Let $R$ be a curve ruling $E$. As $\gamma(D)$ is nef, there exists an $m\in\N$ such that $m\gamma(D)^\vee-R$ is an ample curve. Therefore, by Proposition \ref{prop:hyp su K3}, we produce a rational curve $C$ of class $m\gamma(D)^\vee-R$ which rules an ample divisor, and attach to it a rational tail $R$, so that the connected curve $C+R$ of class $m\gamma(D)^\vee$ rules a divisor and deforms in its Hodge locus by \cite[Corollary 6.3]{Ber20}. By construction, this Hodge locus contains $(X,D^\vee)$, and the result follows.    
\end{proof}

To prove Theorem \ref{thm:density}, we will use the following result of Chen and Lewis on $K3$ surfaces. Let $\sF_g$ be the moduli space of polarized genus $g$ K3 surfaces, and let $\sS_g$ be the universal surface over $\sF_g$. Let $\sC_{g,n}$ be the scheme of relative dimension one  whose fibre over a point $(S,L)\in \sF_g$ consists of all irreducible rational curves contained in $|nL|$. Recall the following result. 
\begin{theorem}[Theorem 1.1, \cite{CL13}]\label{thm:CL13}
    The set $\cup_{n\in\N} \sC_{g,n}$ is dense in the strong topology inside $\sS_g$, for all $g\geq 2$.
\end{theorem}
From this one easily obtains the following.  
\begin{corollary}\label{cor:ratchain}
    Let $S$ be a general projective $K3$ surface. Then any pair of points on $S^{[n]}$ can be arbitrarily approximated by a chain of at most $2n$ rational curves, each of which deforms in a family covering a divisor.
\end{corollary}
\begin{proof}
    Without loss of generality, we can suppose that the two points $\xi_i,\, i\in\{1,2\}$ correspond to reduced subschemes, and that $\supp(\xi_1)\cap \supp(\xi_2)=\emptyset$ otherwise we can take  arbitrarily close approximations by reduced subschemes with such property. Therefore we write 
    $$\xi_i=p^i_1+\dots p^i_n,$$
    with $p^i_1,\ldots, p^i_n$ distinct points on $S$ for $i=1,2$.
    By Theorem \ref{thm:CL13}, we have two ample irreducible curves $R^1_1, R^2_1$ arbitrarily near $p_1^1$ and $p_1^2$ respectively. As these curves are ample, the rational curve $R_1=R^1_1\cup R^2_1$ is connected. Let us consider the rational curve $R_1+p^1_2+\dots p^1_n$ inside $S^{[n]}$: this can be used to approximate the subschemes $p^1_1+p^1_2+\dots p^1_n$ and $p^2_1+p^1_2+\dots p^1_n$. Iterating the argument, one obtains a rational curve (union of two irreducible ample curves) $R_j$ for all $j\in\{1,\dots n\}$ which approximates the two points $p^1_j$ and $p^2_j$. Considering the curve $p^2_1+\dots +p^2_{j-1}+R_j+p^1_{j+1}+\dots + p^1_{n}$ one can approximate the points $p^2_1+\dots +p^2_{j-1}+p^1_j+p^1_{j+1}+\dots + p^1_{n}$ and $p^2_1+\dots +p^2_{j-1}+p^2_j+p^1_{j+1}+\dots + p^1_{n}$. Therefore, by taking the union of these curves we obtain a chain of $2n$ rational irreducible curves which approximate the two points $\xi_1$ and $\xi_2$. By construction, each of these rational curves $C$ deforms in a family which covers the divisor $\{Z\in S^{[n]},\text{ such that } \supp(Z)\cap C\neq\emptyset\}$ and the corollary follows.
\end{proof}

   \begin{proof}[Proof of Theorem \ref{thm:density}]
Let $X$ be a very general IHS manifold in $\sM$. Let $x_1,x_2 \in X$ be two points on it.  
Thanks to \cite[Corollary 1.2]{MPdens} we can pick a point in $\sM$ which parametrizes the punctual Hilbert scheme of a very general projective $K3$ $(S,H)$  arbitrarily close to $X$ and two points $\xi_1, \xi_2\in S^{[n]}$ approximating $x_1$ and $x_2$ respectively.
We take the chain $R$ of $2n$ rational curves approximating $\xi_1$ and $\xi_2$ given by Corollary \ref{cor:ratchain}. 
 We can now apply the Regeneration principle \ref{regprin} to regenerate the union of the divisors ruled by the deformations of the irreducible components of $R$ to obtain a chain of rational curves on $X$ satisfying the statement. 
 
\end{proof}
\begin{remark}
    Actually, using \cite[Theorem 5.5]{CGL} and the Regeneration principle, a simpler version of the proof above yields the existence of infinitely many uniruled divisors for the very general point of {\it any} family $\sX\to B$ of projective IHS manifolds such that one of the fibres is the Hilbert scheme over a $K3$ of odd Picard rank.
\end{remark}

\begin{proof}[Proof of Theorem \ref{thm:bir}]
    To prove the theorem we will  show the existence of an ample uniruled divisor with infinite $\Bir(X)$-orbit.     By \cite[Theorem 1.1]{Og06}, as $\Bir(X)$ is infinite, there exists an element $g\in \Bir(X)$ of infinite order. Let $D$ be an ample uniruled divisor, whose existence is granted by Theorem \ref{thm:to zal}. We claim that the orbit of $D$ via $g$ is infinite, as otherwise a multiple of $g$ would give an isometry of the lattice $D^\perp\subset \NS(X)$.  The latter is negative definite as $D$ is ample and has therefore finite isometry group. Hence $g$ would act with finite order on both $D$ and $D^\perp$, which is absurd  and the claim follows.
    \end{proof}

We recall now the following well-known result for the reader's convenience. This tells us that Theorem \ref{thm:bir} yields its conclusion only for a codimension at least one locus in the moduli space of projective IHS manifolds. 
\begin{lemma}\label{lem:folk}
    Let $X$ be a projective IHS manifold with $\rho(X)=1$. Then $\Aut(X)=\Bir(X)$ and it is a finite group.
\end{lemma}
\begin{proof}
First of all recall that a birational map between two IHS manifolds sending an ample class into an ample class can be extended to an isomorphism. As such, when $\rho(X)=1$, we have $\Aut(X)=\Bir(X)$.     By \cite[Theorem 4.8]{Fuj78}  the group of automorphisms of a compact K\"ahler manifold that fix a K\"ahler class has only finitely many connected components. On the other hand the group of automorphisms of an IHS manifold $X$ is discrete, since $h^0(X,T_X)=h^0(X,\Omega^1_X)=0$. Hence $\Aut(X)$ must be finite. 
\end{proof}

 \bibliography{literature}
\bibliographystyle{alpha}

\end{document}